\documentclass[11pt]{article}
\usepackage[margin=1in]{geometry} 
\usepackage{amssymb,amsfonts,amsmath,bbm,mathrsfs,stmaryrd,mathtools}
\usepackage{xcolor}
\usepackage{url}

\usepackage{enumerate}

\usepackage{hyperref}
\hypersetup{colorlinks,
             linkcolor=black!75!red,
             citecolor=blue,
             pdftitle={},
             pdfproducer={pdfLaTeX},
             pdfpagemode=None,
             bookmarksopen=true
             bookmarksnumbered=true}

\usepackage{tikz}
\usetikzlibrary{arrows,calc,decorations.pathreplacing,decorations.markings,intersections,shapes.geometric,through,fit,shapes.symbols,positioning,decorations.pathmorphing}

\usepackage{braket}

\usepackage[amsmath,thmmarks,hyperref]{ntheorem}
\usepackage{cleveref}

\creflabelformat{enumi}{#2(#1)#3}

\crefname{section}{Section}{Sections}
\crefformat{section}{#2Section~#1#3} 
\Crefformat{section}{#2Section~#1#3} 

\crefname{subsection}{\S}{\S\S}
\AtBeginDocument{%
  \crefformat{subsection}{#2\S#1#3}%
  \Crefformat{subsection}{#2\S#1#3}% 
}

\crefname{subsubsection}{\S}{\S\S}
\AtBeginDocument{%
  \crefformat{subsubsection}{#2\S#1#3}%
  \Crefformat{subsubsection}{#2\S#1#3}% 
}

\theoremstyle{plain}

\newtheorem{lemma}{Lemma}[section]
\newtheorem{proposition}[lemma]{Proposition}

\newtheorem{theorem}[lemma]{Theorem}

\newtheorem{fact}[lemma]{Fact}

\theoremstyle{nonumberplain}
\newtheorem{theoremN}{Theorem}
\newtheorem{propositionN}{Proposition}

\theoremstyle{plain}
\theorembodyfont{\upshape}
\theoremsymbol{\ensuremath{\blacklozenge}}

\newtheorem{definition}[lemma]{Definition}
\newtheorem{example}[lemma]{Example}
\newtheorem{remark}[lemma]{Remark}

\newtheorem{construction}[lemma]{Construction}

\crefname{definition}{definition}{definitions}
\crefformat{definition}{#2definition~#1#3} 
\Crefformat{definition}{#2Definition~#1#3} 

\crefname{ex}{example}{examples}
\crefformat{example}{#2example~#1#3} 
\Crefformat{example}{#2Example~#1#3} 

\crefname{remark}{remark}{remarks}
\crefformat{remark}{#2remark~#1#3} 
\Crefformat{remark}{#2Remark~#1#3} 

\crefname{convention}{convention}{conventions}
\crefformat{convention}{#2convention~#1#3} 
\Crefformat{convention}{#2Convention~#1#3} 

\crefname{notation}{notation}{notations}
\crefformat{notation}{#2notation~#1#3} 
\Crefformat{notation}{#2Notation~#1#3} 

\crefname{table}{table}{tables}
\crefformat{table}{#2table~#1#3} 
\Crefformat{table}{#2Table~#1#3}

\crefname{lemma}{lemma}{lemmas}
\crefformat{lemma}{#2lemma~#1#3} 
\Crefformat{lemma}{#2Lemma~#1#3} 

\crefname{proposition}{proposition}{propositions}
\crefformat{proposition}{#2proposition~#1#3} 
\Crefformat{proposition}{#2Proposition~#1#3} 

\crefname{corollary}{corollary}{corollaries}
\crefformat{corollary}{#2corollary~#1#3} 
\Crefformat{corollary}{#2Corollary~#1#3} 

\crefname{theorem}{theorem}{theorems}
\crefformat{theorem}{#2theorem~#1#3} 
\Crefformat{theorem}{#2Theorem~#1#3} 

\crefname{enumi}{}{}
\crefformat{enumi}{(#2#1#3)}
\Crefformat{enumi}{(#2#1#3)}

\crefname{assumption}{assumption}{Assumptions}
\crefformat{assumption}{#2assumption~#1#3} 
\Crefformat{assumption}{#2Assumption~#1#3} 

\crefname{fact}{fact}{Facts}
\crefformat{fact}{#2fact~#1#3} 
\Crefformat{fact}{#2Fact~#1#3} 

\crefname{construction}{construction}{Constructions}
\crefformat{construction}{#2construction~#1#3} 
\Crefformat{construction}{#2Construction~#1#3} 

\crefname{equation}{}{}
\crefformat{equation}{(#2#1#3)} 
\Crefformat{equation}{(#2#1#3)}

\numberwithin{equation}{section}

\theoremstyle{nonumberplain}
\theoremsymbol{\ensuremath{\blacksquare}}

\newtheorem{proof}{Proof}
\newcommand\pf[1]{\newtheorem{#1}{Proof of \Cref{#1}}}

\newcommand\bA{{\mathbb A}}
\newcommand\bB{{\mathbb B}}

\newcommand\bD{{\mathbb D}}
\newcommand\bE{{\mathbb E}}
\newcommand\bF{{\mathbb F}}
\newcommand\bG{{\mathbb G}}
\newcommand\bH{{\mathbb H}}

\newcommand\bK{{\mathbb K}}

\newcommand\bM{{\mathbb M}}
\newcommand\bN{{\mathbb N}}

\newcommand\bR{{\mathbb R}}
\newcommand\bS{{\mathbb S}}
\newcommand\bT{{\mathbb T}}
\newcommand\bU{{\mathbb U}}

\newcommand\bZ{{\mathbb Z}}

\newcommand\cH{{\mathcal H}}

\newcommand\cL{{\mathcal L}}

\newcommand\cO{{\mathcal O}}

\newcommand\fg{{\mathfrak g}}
\newcommand\fh{{\mathfrak h}}

\newcommand\fn{{\mathfrak n}}

\newcommand{\qedhere}{\mbox{}\hfill\ensuremath{\blacksquare}}

\renewcommand{\square}{\mathrel{\Box}}

\title{Finite central extensions of type I}
\author{Alexandru Chirvasitu}

\begin{document}

\date{}

\newcommand{\Addresses}{{% additional braces for segregating \footnotesize
  \bigskip
  \footnotesize

  \textsc{Department of Mathematics, University at Buffalo, Buffalo,
    NY 14260-2900, USA}\par\nopagebreak \textit{E-mail address}:
  \texttt{achirvas@buffalo.edu}

% %   \medskip
% %   
% %   \textsc{Department of Mathematics, institution,
% %     address}\par\nopagebreak \textit{E-mail address}:
% %   \texttt{??}
% % 
}}

\maketitle

\begin{abstract}
  Let $\mathbb{G}$ be a Lie group with solvable connected component and finitely-generated component group and $\alpha\in H^2(\mathbb{G},\mathbb{S}^1)$ a cohomology class. We prove that if $(\mathbb{G},\alpha)$ is of type I then the same holds for the finite central extensions of $\mathbb{G}$. In particular, finite central extensions of type-I connected solvable Lie groups are again of type I. This is by contrast with the general case, whereby the type-I property does not survive under finite central extensions.

  We also show that ad-algebraic hulls of connected solvable Lie groups operate on these even when the latter are not simply connected, and give a group-theoretic characterization of the intersection of all Euclidean subgroups of a connected, simply-connected solvable group $\mathbb{G}$ containing a given central subgroup of $\mathbb{G}$.
\end{abstract}


\vspace{.5cm}


\tableofcontents

%%%%%%%%%%%%%%%%%%%%%%%%%%%%%%%%%%%%%%%%%%%%%%%%%%%%%%%%%%%%%%%%%%%%%%%%%%%%%
%%%%%%%%%%%%%%%%%%%%%%%%%%%%%%%%%%%%%%%%%%%%%%%%%%%%%%%%%%%%%%%%%%%%%%%%%%%%%
\section*{Introduction}

The proximal motivator for the present paper is the observation, evidenced by \cite[Example 20]{be}, that finite central extensions of type-I groups need not, in general, be type-I. By way of unwinding this remark, recall \cite[\S IV.1]{br-coh} that an {\it extension} of a group $\bG$ by another, $\bF$, is an exact sequence
\begin{equation}\label{eq:feg-intro}
  \{1\}\to \bF\to \bE\to \bG\to \{1\}
\end{equation}
(not infrequently, authors refer to this as an extension {\it of} $\bF$ {\it by} $\bG$ instead: \cite[\S 11.1]{rob-gr}, \cite[\S 9.1]{rot}, etc.; here, we observe the stated convention). The extension is {\it central} if $\bF\le \bE$ is a central subgroup, and {\it finite} if $\bF$ is finite.

As for the other ingredient, a locally compact (Hausdorff, second-countable) group $\bG$ is {\it of type I} if, for every irreducible unitary representation
\begin{equation*}
  \rho:\bG\to U(\cH)\quad(=\text{unitary group of the Hilbert space $\cH$})
\end{equation*}
the $C^*$-algebra generated by the unitary operators $\rho(g)$, $g\in \bG$ contains the compact operators on $\cH$.

This is one of many characterizations, and by no means the most edifying or inspiring: see e.g. \cite[Theorem 7.6]{folland} for a convenient listing and the surrounding text for pointers to the vast literature. Equivalent conditions require
\begin{itemize}
\item that various naturally-defined Borel structures on the {\it unitary dual} $\widehat{\bG}$ (isomorphism classes of irreducible unitary $\bG$-representations) be {\it countably separated} or, alternatively, {\it standard} (\cite[Theorem 6.8.7]{ped-aut} is also illuminating here);
\item or that the {\it Fell topology} \cite[\S 7.2]{folland} on $\widehat{\bG}$ satisfy the $T_0$ separation axiom;
\item or that an irreducible unitary representation be uniquely determined by its corresponding {\it primitive ideal} in the {\it full group $C^*$-algebra $C^*(\bG)$} \cite[discussion following Corollary 7.2]{folland};
\end{itemize}
and so on. The concept formalizes the intuition of having a reasonable moduli space of irreducible representations; the pleasant character of $\widehat{\bG}$ for type-I groups allows for a theory of integration over $\widehat{\bG}$, and hence to well-behaved decompositions of arbitrary unitary representations as {\it direct integrals} of irreducible ones (\cite[\S 7.4]{folland}, \cite[\S 18.8]{dixc}, etc.).

Broad classes of groups are known to be of type I; a sample: connected Lie if either semisimple or nilpotent, connected real algebraic, discrete virtually abelian (i.e. having an abelian subgroup of finite index) \cite[Theorem 7.8]{folland}, linear algebraic groups over characteristic-0 local fields \cite[Theorem 2]{be}.

That as mild and reasonable an operation as a central extension could break the type-I property might be somewhat surprising. It seemed natural, in light of this, to identify some sufficient properties on the type-I group $\bG$ (being centrally extended) that will prevent this from happening.

For one thing, \cite[Example 20]{be} is discrete: it is, in fact, a finite central extension of an infinitely-generated discrete abelian group. At the opposite extreme from this, one might hope for positive results in working with connected Lie groups instead. With this motivation, \Cref{th:solv} below (somewhat attenuated here for brevity) reads

\begin{theoremN}
  If a Lie group $\bG$ with solvable connected component $\bG_0$ and finitely-generated connected group $\bG/\bG_0$ is of type-I, so are its finite central extensions. \qedhere
\end{theoremN}

Connected solvable Lie groups are well understood from the perspective of characterizing those of type I: \cite{ak} and \cite[Chapter IV]{puk} for instance achieve this completely in the {\it simply-connected} case and \cite{am} provides much illuminating discussion and adjacent/partial results. The qualification of simple connectedness matters however: sources tend to discuss mostly simply-connected groups, whereas
\begin{itemize}
\item $\bE$ being a covering space of $\bG$, finite central extensions \Cref{eq:feg-intro} are only truly interesting for {\it non-}simply-connected groups;  
\item and it is perfectly possible for a non-simply-connected solvable Lie group to be of type I without its universal cover being so \cite{dix-revet}. 
\end{itemize}

With this in mind, it seems pertinent to revisit some of the techniques employed in the cited work and assess how crucial simple connectedness is. The background is reviewed more fully in \Cref{subse:nonsc}, but the crux of the matter appears to be that one needs the {\it algebraic hull} $Ad(\bG)_{alg}\le GL(\fg)$ to operate as an automorphism group of $\bG$.

Here, $\fg$ is the Lie algebra of $\bG$, $Ad$ is the adjoint representation of $\bG$ on that Lie algebra, and the algebraic hull is the smallest {\it algebraic} \cite[Chapitre II, \S 1, D\'efinition 1]{chv2} Lie group containing the image of that representation.

The algebraic hull is certainly (by definition) an automorphism group of the Lie {\it algebra} $\fg$, and that action lifts over to one on $\bG$ itself when the latter is simply connected. In \Cref{th:allfix} below we argue that in fact simple connectedness (for solvable $\bG$) is not, in fact, necessary:

\begin{theoremN}
  The algebraic hull $Ad(\bG)_{alg}$ of a connected, simply-connected solvable Lie group fixes the center $Z(\bG)$ pointwise.

  In particular, the action of $Ad(\bG)_{alg}$ on $\bG$ descends to one on any Lie group covered by $\bG$.  \qedhere
\end{theoremN}

This does dot appear to be immediate, and requires an analysis of how central subgroups are positioned inside a connected, simply-connected solvable Lie group $\bG$ (in the spirit of \cite{chv-solv}, say). Specifically, it becomes important to understand the families of {\it Euclidean} subgroups of $\bG$ containing a given central subgroup. As an offshoot of that line of reasoning, we introduce a kind of ``saturation'' procedure for central subgroups and prove in \Cref{pr:purify} that it coincides with what might be termed the {\it Euclidean hull} of a central subgroup (a bit of a misnomer: that hull is not itself Euclidean, in general!):

\begin{propositionN}
  For a central subgroup $\bA\le \bG$ of a connected, simply-connected solvable Lie group the following subgroups of $\bG$ coincide:
  \begin{enumerate}[(a)]
  \item\label{item:15} the {\it purification} of $\bA$ in the center $Z(\bG)$, i.e. the largest $\overline{\bA}\le Z(\bG)$, containing $\bA$, and such that $Z(\bG)/\overline{\bA}$ is torsion-free;
  \item\label{item:16} the intersection of all Euclidean subgroups of $\bG$ containing $\bA$;
  \item\label{item:17} the intersection of all conjugates of any one minimal-dimensional Euclidean subgroup of $\bG$ containing $\bA$.  \qedhere
  \end{enumerate}  
\end{propositionN}

\Cref{se:fail} contains some odds and ends on the theme of type-I pathology (e.g. a natural example of a type-I semidirect product of a non-type-I group by a compact group, all Lie and connected: see \Cref{pr:compmaut1}and \Cref{ex:adm1}).

%%%%%%%%%%%%%%%%%%%%%%%%%%%%%%%%%%%%%%%%%%%%%%%%%%%%%%%%%%%%%%%%%%%%%%%%%%%%%
\subsection*{Acknowledgements}

This work is partially supported by NSF grant DMS-2001128.

I am grateful for valuable input on the present material from B. Bekka, S. Echterhoff and KH Neeb.

%%%%%%%%%%%%%%%%%%%%%%%%%%%%%%%%%%%%%%%%%%%%%%%%%%%%%%%%%%%%%%%%%%%%%%%%%%%%%
%%%%%%%%%%%%%%%%%%%%%%%%%%%%%%%%%%%%%%%%%%%%%%%%%%%%%%%%%%%%%%%%%%%%%%%%%%%%%
\section{Preliminaries}\label{se.prel}

Throughout, $\widehat{\bG}$ denotes the {\it spectrum} (or {\it unitary dual}) of a locally compact group $\bG$ (as in \cite[\S 18.3.2]{dixc}, say), i.e. the set of isomorphism classes of irreducible unitary representations. We will often have to speak of type-I {\it regularly embedded} normal subgroups $\bN\trianglelefteq \bG$ in the sense of \cite[\S 3.8]{mack-unit}:
\begin{itemize}
\item the Borel structure on the quotient $\widehat{\bN}/\bG$ is countably separated;
\item or $\bG$-ergodic measures on $\widehat{\bN}$ are supported on orbits;
\item or any number of other equivalent formulations \cite[Theorem 1]{glm}, all equivalent in the cases of interest.
\end{itemize}
Alternative terminology: $\bG$ (or $\bG/\bN$) {\it acts smoothly} on $\widehat{\bN}$.

For a connected Lie group $\bG$ write
\begin{itemize}
\item $\bG_0$ for its {\it connected component} (i.e. the connected component containing the identity); 
\item $\widetilde{\bG}$ for its universal cover;
\item and $N(\bG)$ for its {\it nilradical}: the largest connected, normal, nilpotent subgroup \cite[\S III.9.7, Proposition 23]{bourb-lie-13}. 
\end{itemize}

We gather a number of well-known observations about the center of a Lie group, which it will be convenient to bear in mind.

\begin{remark}\label{re:allaboutcenters}
  Let $\bG$ be a connected Lie group throughout.
  \begin{enumerate}[(a)]
  \item\label{item:3} In general, its center $Z(\bG)$ is a sum
    \begin{equation*}
      Z(\bG)\cong \bR^d\oplus \bT^e\oplus \bD
    \end{equation*}
    where $\bT^e$ denotes the $e$-dimensional torus and $\bD$ is finitely-generated discrete abelian (and hence in turn, of the form $\bF\oplus \bZ^f$ for finite $\bF$). This follows from \cite[Corollary 4.2.6]{de}.
  \item\label{item:4} When $\bG$ is moreover simply-connected the torus portion is absent: $Z(\bG)_0\cong \bR^d$. This is because normal, connected (hence also closed by \cite[discussion preceding \S XIII of Chapter IV]{chv1}) Lie subgroups $\bN\trianglelefteq \bG$ of simply-connected Lie groups are again simply-connected.
  \item\label{item:5} That last remark follows for instance from the long exact homotopy sequence
    \begin{equation*}
      \cdots\to \pi_2(\bG/\bN)\to \pi_1(\bN)\to \pi_1(\bG)\to \cdots
    \end{equation*}
    attached to the inclusion $\bN\le \bG$ \cite[\S 17.3]{steen} and the fact that $\pi_2$ vanishes for connected Lie groups (e.g. because $\pi_1$ vanishes for their based-loop spaces \cite[Theorem 21.7 and Remark 2 following it]{milnor-morse}).
  \item\label{item:6} Still assuming $\bG$ simply-connected, the connected component $Z(\bG)_0$ is precisely the intersection $Z(\bG)\cap N(\bG)$ with the nilradical. One inclusion is obvious, while the connectedness of $Z(\bG)\cap N(\bG)$ follows from a familiar argument (e.g. as in the proofs of \cite[\S III.9.5, Propositions 15 and 16]{bourb-lie-13}): 

    Any group $\Gamma$ of automorphisms of $\bN:=N(\bG)$ operates on the Lie algebra $\fn:=Lie(\bN)$. Since the exponential map
    \begin{equation}\label{eq:expmap}
      \exp:\fn\to \bN
    \end{equation}
    is an analytic isomorphism \cite[\S III.9.5, Proposition 13]{bourb-lie-13}, a non-trivial element is fixed by $\Gamma$ if and only if the unique one-parameter subgroup of $\bN$ containing it is. In conclusion, the fixed-point set of $\Gamma$ in $\bN$ is a {\it connected} closed subgroup of the latter. Applying this to the group of inner automorphisms induced by $\bG$, we obtain
    \begin{equation*}
      Z(\bG)_0 = Z(\bG)\cap \bN.
    \end{equation*}
    All in all, for a connected, simply-connected Lie group we have a decomposition
    \begin{equation}\label{eq:zsplit}
      Z(\bG)=\bZ(\bG)_0\oplus \bD_{nil'},\quad \bZ(\bG)_0 = Z(\bG)\cap N(\bG)\cong \bR^d,\quad \bD_{nil'}\text{ discrete}. 
    \end{equation}
    While the connected component of the center $Z(\bG)$ is of course canonical, the complementary summand $\bD_{nil'}$ is not, in general: one can always deform generators thereof by adding elements of $Z(\bG)_0$.
    
  \item\label{item:7} If $\bG$ is simply-connected and {\it solvable} the discrete piece $\bD_{nil'}$ in \Cref{eq:zsplit} is free abelian (i.e. torsion-free) \cite[Theorem 1]{chv-solv}.
  \end{enumerate}
\end{remark}

Because at this point it seems pertinent as a follow-up remark, note that the ``unique one-parameter group'' mentioned in \Cref{re:allaboutcenters} \Cref{item:6} can be made sense of somewhat more broadly.

\begin{lemma}\label{le:solvsc1par}
  Let $\bG$ be a connected, simply-connected solvable Lie group and $\bN:=N(\bG)$ its nilradical.

  An element of $\bN$ lies on a unique 1-parameter subgroup of $\bG$.
\end{lemma}
\begin{proof}
  The isomorphism \Cref{eq:expmap} ensures that it lies on a unique such subgroup of $\bN$, but the claim is that no 1-parameter group of $\bG$ {\it not} contained in $\bN$ can intersect the latter non-trivially.

  Suppose it did: for some element $v\in \fg:=Lie(\bG)$ of the Lie algebra of $\bG$, not contained in the smaller Lie algebra $\fn:=Lie(\bN)$, 
  \begin{equation*}
    \bR\ni t\mapsto \exp(tv)\in \bG
  \end{equation*}
  intersects $\bN$ at some $t\ne 0$. The Lie algebra $\fn\oplus \bR v\le \fg$ is that of a closed, connected, simply-connected Lie subgroup of $\bG$ (\cite[\S II, Theorem]{chv-solv} or \cite[\S III.9.6, Proposition 21]{bourb-lie-13}). In restricting attention to that subgroup, we can assume that $\bN\le \bG$ has codimension 1.

  By assumption, $\bG/\bN$ is non-trivial and compact, i.e. a circle. But it is also a quotient of a simply-connected Lie group by a normal, closed, connected subgroup, so must be simply-connected (\Cref{re:allaboutcenters} \Cref{item:4}); this is the desired contradiction.
\end{proof}

Solvability cannot be dropped in \Cref{le:solvsc1par}: 

\begin{example}
  Consider a connected, simply-connected, semisimple Lie group $\bK$ (e.g. a special unitary group $SU(n)$) acting on a real vector space $\bR^d$ so that some circle $\bS^1\le \bK$ fixes a vector $v\in \bR^d$.

  In the simply-connected group $\bG:=\bR^d\rtimes \bK$ the one-parameter groups generated by $v\in Lie(\bR^d)\cong \bR^d$ and by $v+w$ for a generator $w$ of the Lie algebra of the circle in question both contain $v\in \bR^d\le N(\bG)$.
\end{example}

On the other hand, even when $\bG$ is (as in the statement) solvable, \Cref{le:solvsc1par} does not apply to elements {\it outside} the nilradical: see \Cref{re:centonexp}.

%%%%%%%%%%%%%%%%%%%%%%%%%%%%%%%%%%%%%%%%%%%%%%%%%%%%%%%%%%%%%%%%%%%%%%%%%%%%%
%%%%%%%%%%%%%%%%%%%%%%%%%%%%%%%%%%%%%%%%%%%%%%%%%%%%%%%%%%%%%%%%%%%%%%%%%%%%%
\section{Discrete / abelian groups}\label{se:discrete}

Since we are concerned with the survival of the type-I property under finite central extensions, it will be convenient to have some shorthand language for the notions involved.

\begin{definition}\label{def:robust}
  Let $\bG$ be a locally compact, second countable group.
  \begin{itemize}
  \item If $\bG$ is type-I is {\it fc-robust} (for `finite central') or just {\it robust} if all of its finite central extensions are of type I. We might also say that the group is {\it (fc-)robustly type-I}.
  \item The same terms apply to pairs $(\bG,\alpha)$ of a group $\bG$ and a {\it cohomology class} $\alpha$ (\cite[Chapter I, discussion preceding Proposition 5.2]{am}): the question is whether, given an element $\alpha$ of the cohomology group $H^2(\bG,\bS^1)$ \cite[Part I, Chapter I]{moore-ext}, the middle term in the central extension
    \begin{equation*}
      \{1\}\to \bS^1\to \bullet\to \bG\to \{1\}
    \end{equation*}
    attached to $\alpha$ (\cite[Part I, Introduction]{moore-ext}) is type-I. We abbreviate this to saying that the pair $(\bG,\alpha)$ is of type I.

    The pair $(\bG,\alpha)$ is {\it fc-robustly} (or just plain {\it robustly}) type-I if for every finite central extension $\bE\to \bG$, regarding $\alpha$ as an element of $H^2(\bE,\bS^1)$ through the restriction map
    \begin{equation*}
      H^2(\bG,\bS^1)\to H^2(\bE,\bS^1),
    \end{equation*}
    the pair $(\bE,\alpha)$ is of type I.    
  \end{itemize}
\end{definition}

\begin{proposition}\label{pr:fgfc}
  Let
  \begin{equation}\label{eq:fea}
    \{1\}\to \bF\to \bE\to \bA\to \{1\}
  \end{equation}
  be a finite central extension of discrete groups. If $\bA$ is finitely-generated and of type I then so is $\bE$.
\end{proposition}
\begin{proof}
  Since $\bA$ is assumed type-I it has a finite-index abelian subgroup by Thoma's theorem (see e.g. \cite[Satz 6]{thoma}, \cite[Theorem 1]{tt} or \cite[Theorem E]{bekka-count}). Passing to a cofinite subgroup does not affect the type-I property, so we may as well assume that $\bA$ is abelian to begin with.

  We now have a central extension of an {\it abelian} group, whence the commutator map
  \begin{equation*}
    \bE^2\ni (x,y)\xmapsto[]{} [x,y]:=xyx^{-1}y^{-1}\in \bF
  \end{equation*}
  descends to a (skew-symmetric) bilinear map $\omega:\bA^2\to \bF$.

  The map
  \begin{equation*}
    \omega(a,-):\bA\to \bF
  \end{equation*}
  annihilates a cofinite subgroup $\ker_{\omega}(a)$ of $\bA$, and the (finite!) intersection
  \begin{equation*}
    \bA_0:=\bigcap_{\text{generators }a}\ker_{\omega}(a)\le \bA
  \end{equation*}
  will then be a finite-index subgroup of $\bA$ in the kernel of $\omega$, in the sense that
  \begin{equation*}
    \omega(\bA_0,-)\equiv 0\equiv \omega(-,\bA_0).
  \end{equation*}
  It follows that the central extension \Cref{eq:fea} can be rearranged as
  \begin{equation*}
    \{1\}\to \bF\times \bA_0\to \bE\to \bA/\bA_0\to \{1\},
  \end{equation*}
  which is of course virtually abelian. 
\end{proof}

\begin{remark}
\cite[Example 20]{be} shows that the finite generation assumption in \Cref{pr:fgfc} cannot, in general, be dropped: the central extension discussed there is of the form
  \begin{equation*}
    \{1\}\to \bZ/p\to \bE\to V\times V\to \{1\}
  \end{equation*}
  for a prime $p$, where $V$ is a direct sum of countably infinitely many copies of $\bZ/p$. It is obtained by summing infinitely many copies of the order-$p^3$ Heisenberg group and identifying the centers to a single copy of $\bZ/p$.

  More generally, some such pathological extension
  \begin{equation*}
    \{1\}\to \bF\to \bE\to \bA\to \{1\}
  \end{equation*}
  exists whenever there is a skew-symmetric bilinear map $\omega:\bA^2\to \bF$ whose kernel
  \begin{equation*}
    \{a\in \bA\ |\ \omega(a,-)\equiv 0\equiv \omega(-,a)\}
  \end{equation*}
  has infinite index in $\bA$.
\end{remark}

We revisit \Cref{pr:fgfc}, this time assuming abelianness but dropping discreteness. We also twist by a cocycle, as in \Cref{def:robust}. 

\begin{theorem}\label{th:cpctgenab}
  Let $\bA$ be a locally compact, abelian, compactly generated group and $\alpha\in H^2(\bA,\bS^1)$. The pair $(\bA,\alpha)$ is robustly type-I in the sense of \Cref{def:robust}.
\end{theorem}

We fix a finite central extension \Cref{pr:fgfc} throughout, and begin with some simplifying assumptions.

\begin{lemma}\label{le:easiere}
  Let $\bE$ be a locally compact group fitting into a finite central extension \Cref{eq:fea} with $\bA$ abelian and compactly generated.

  Modulo replacing $\bE$ with a finite-index subgroup, we can assume it splits as
  \begin{equation}\label{eq:edecomp}
      \bE\cong \bR^d \times \bZ^e\times \bK_{\bE}
    \end{equation}
    with compact abelian $\bK_{\bE}$ for a finite extension
    \begin{equation}\label{eq:keext}
      \{1\}\to \bF\to \bK_{\bE}\to \bK\to \{1\}.
    \end{equation}
\end{lemma}
\begin{proof}
  The standard structure result on compactly-generated locally compact abelian groups (e.g. \cite[Theorem 4.2.2]{be}) says that we have a decomposition
  \begin{equation*}
    \bA\cong \bR^d\times \bZ^e\times \bK
  \end{equation*}
  for compact abelian $\bK$. This will be the $\bK$ in the statement, perhaps after passing to a cofinite subgroup. We effect the splitting \Cref{eq:edecomp} gradually.
  
  \begin{enumerate}[(1)]
  \item {\bf Splitting off the Euclidean factor.} Note first that finite central extensions
    \begin{equation*}
      \{1\}\to \bF\to \bullet\to \bR^d\to \{1\}
    \end{equation*}
    of Euclidean groups always split: this is because, $\bR^d$ being connected, the commutator induced bilinear form $\bR^d\times \bR^d\to \bF$ (as in the proof of \Cref{pr:fgfc}) is trivial. It follows then that $\bullet$ is abelian, whence the applicability of the usual structure theorems for locally compact abelian groups \cite[Theorem 4.2.2]{de}. It already follows from this that $\bE$ fits into an extension
    \begin{equation*}
      \{1\}\to \bR^d\times \bF\to \bE\to \bZ^e\times \bK\to \{1\},
    \end{equation*}
    so the quotient $\bE\to \bR^d$ splits and we have a semidirect-product decomposition
    \begin{equation*}
      \bE\cong \bE'\rtimes \bR^d
    \end{equation*}
    for
    \begin{equation}\label{eq:e'}
      \{1\}\to \bF\to \bE'\to \bZ^e\times \bK\to \{1\}. 
    \end{equation}
    Because $\bR$
    \begin{itemize}
    \item operates trivially on $\bF$;
    \item and on $\bZ^e\times \bK$;
    \item and is connected, 
    \end{itemize}
    it must operate trivially period. In short: $\bE\cong \bE'\times \bR^d$.
  \item {\bf Splitting off a free abelian factor.} Next, note that $|\bF|\bZ^e$ is contained in the kernel of the commutator bilinear form
    \begin{equation*}
      (\bZ^e\times \bK)^2\to \bF
    \end{equation*}
    induced by \Cref{eq:e'}. Upon shrinking $\bE'$ (and $\bE$) to a cofinite subgroup we can assume $\bZ^e$ itself is in that kernel. The subgroup
    \begin{equation*}
      \{1\}\to \bF\to \bullet\to\bZ^e\to \{1\}
    \end{equation*}
    of $\bE'$ is then abelian finitely generated, and hence the extension in question splits: $\bullet\cong \bF\times \bZ^e$. In any case, the surjection $\bE'\to \bZ^e$ also splits, so that $\bE'\cong \bK_{\bE}\rtimes \bZ^e$ with \Cref{eq:keext}. Moreover, $\bZ^e$-action on $\bK_{\bE}$ is compatible with that extension: it centralizes $\bF$ and also centralizes the quotient $\bK$.

    But then $|\bF|\bZ^e$ acts trivially on $\bK_{\bE}$, hence the decomposition \Cref{eq:edecomp} after, perhaps, shrinking to an even smaller (but still cofinite) subgroup.
  \item {\bf $\bK_{\bE}$ can be assumed abelian.} So far we only have a finite central extension \Cref{eq:keext} with compact abelian $\bK$. As usual, that extension gives rise to a bilinear form
    \begin{equation}\label{eq:k2f}
      \bK^2\to \bF
    \end{equation}
    or equivalently, to a morphism
    \begin{equation*}
      \bK\to \mathrm{Hom}(\bK,\bF). 
    \end{equation*}
    Since the codomain is discrete and the domain compact, that morphism is trivial on a cofinite subgroup $\bK'\le \bK$. But then, restricting the commutator bilinear form \Cref{eq:k2f} vanishes when restricted to $\bK'\times \bK'$, so the cofinite subgroup
    \begin{equation*}
      \{1\}\to \bF\to \bullet\to \bK'\to \{1\}
    \end{equation*}
    of $\bK_{\bE}$ is abelian.
  \end{enumerate}
  This finishes the proof. 
\end{proof}

A general piece of notation, before moving on to the proof of \Cref{th:cpctgenab}: for a cohomology class $\alpha\in H^2(\bG,\bA)$ we write $\bG_{\alpha}$ for the middle term of the resulting extension
\begin{equation}\label{eq:galpha}
  \{1\}\to \bA\to \bG_{\alpha}\to \bG\to \{1\}.
\end{equation}

\pf{th:cpctgenab}
\begin{th:cpctgenab}
  \Cref{le:easiere} allows us to assume that we have a decomposition \Cref{eq:edecomp} for a compact abelian group $\bK_{\bE}$ fitting into the extension \Cref{eq:keext}, so that
  \begin{equation*}
    \bA\cong \bR^d\times \bZ^e\times \bK. 
  \end{equation*}
  We are assuming that the central extension
  \begin{equation*}
    \{1\}\to \bS^1\to \bA_{\alpha}\to \bA\to \{1\}
  \end{equation*}
  attached to $\alpha\in H^2(\bA,\bS^1)$ is of type I, and seek to prove the same of $\bE_{\alpha}$ in
  \begin{equation*}
    \{1\}\to \bS^1\to \bE_{\alpha}\to \bE\to \{1\}. 
  \end{equation*}
  In both $\bA_{\alpha}$ and $\bE_{\alpha}$ we can isolate the lattice $\bE^e$ as a quotient, as in
  \begin{align*}
    \{1\}\to \bN\to \bA_{\alpha}\to \bZ^e\to \{1\}\\
    \{1\}\to \bN_{\bE}\to \bE_{\alpha}\to \bZ^e\to \{1\}\\
  \end{align*}
  for central extensions
  \begin{align*}
    \{1\}\to \bS^1\to \bN\to \bR^d\times \bK\to \{1\}\\
    \{1\}\to \bS^1\to \bN_{\bE}\to \bR^d\times \bK_{\bE}\to \{1\}.\\
  \end{align*}
  Next, observe that $\bN$ and $\bN_{\bE}$ are of type I regardless of any other considerations: they both contain
  \begin{itemize}
  \item connected nilpotent (hence type-I \cite[Th\'eor\`eme 3]{dix-nil-1}) normal subgroups $\bN$ and $\bN_{\bE}$;
  \item with compact quotients $\bK$ and $\bK_{\bE}$ respectively, hence the conclusion by \cite[Corollary 4.5]{gk} (for instance).
  \end{itemize}
  Because $\bA_{\alpha}$ and $\bE_{\alpha}$ contain the type-I normal subgroups $\bN$ and $\bN_{\bE}$ respectively with abelian discrete quotients $\bZ^e$, \cite[Chapter II, Corollary to Theorem 9 and Chapter I, Proposition 10.4]{am} imply that in both cases being type-I amounts to the following two conditions. 
  \begin{itemize}
  \item $\bN$ or $\bN_{\bE}$ are regularly embedded in $\bA_{\alpha}$ and $\bE_{\alpha}$ respectively;
  \item and (leaving the entirely analogous statement for $\bN_{\bE}$ to the reader) for every $x\in \widehat{\bN}$, if
    \begin{equation*}
      \bA_{\alpha,x}\le \bA_{\alpha},\quad \bZ^e_x\le \bZ^e
    \end{equation*}    
    denote the isotropy groups of $x$ and
    \begin{equation*}
      \alpha_x\in H^2(\bA_{\alpha,x}/\bN,\bS^1)\cong H^2(\bZ^e_x,\bS^1)
    \end{equation*}
    is the {\it Mackey obstruction class} attached to $x$ \cite[Chapter I, Proposition 10.3]{am}, the pair $(\bZ^e_x,\alpha_x)$ is of type I.
  \end{itemize}

  I now claim that the spectrum $\widehat{\bN_{\bE}}$ decomposes (non-canonically) as a disjoint union
  \begin{equation}\label{eq:nedecomp}
    \widehat{\bN_{\bE}}\cong \coprod_{\widehat{\bF}} \widehat{\bN}
  \end{equation}
  of copies of the spectrum of $\bN$ so as to preserve all of the relevant structure: the $\bZ^e$-action as well as the resulting Mackey obstructions. Clearly, given the preceding discussion, such a decomposition would entail the conclusion. It thus remains to explain how \Cref{eq:nedecomp} comes about; the rest of the proof is devoted to precisely this task.
  
  As a first step, we do have a decomposition
  \begin{equation*}
    \widehat{\bN_{\bE}} = \coprod_{\chi\in \widehat{\bF}}\widehat{\bN_{\bE}}_{\chi},
  \end{equation*}
  with $\widehat{\bN}_{\bE,\chi}$ denoting those irreducible $\bN_{\bE}$-representations wherein $\bF$ acts via the character $\chi$ (it must act by scalars, being central). For each $\chi\in\widehat{\bF}$ choose (non-canonically) an extension $\overline{\chi}\in \widehat{\bK_{\bE}}$, as the surjection
  \begin{equation*}
    \widehat{\bK_{\bE}} \twoheadrightarrow \widehat{\bF}
  \end{equation*}  
  of discrete abelian groups makes possible. Regard $\overline{\chi}$ as a character on all of $\bN_{\bE}$ by composing with the surjection $\bN_{\bE}\to \bK_{\bE}$.

  Finally, tensoring with the 1-dimensional representation $\overline{\chi}^{-1}$ now induces the desired identification
  \begin{equation*}
    \widehat{\bN}_{\bE,\chi} \cong \widehat{\bN}. 
  \end{equation*}
  The equivariance of this identification under the $\bZ^e$-action follows from the fact that we are twisting with characters $\overline{\chi}:\bN_{\bE}\to \bS^1$ trivial on the central circle, and the fact that the Mackey obstructions are preserved follows by direct examination of the construction (e.g. in \cite[Chapter I, proof of Proposition 10.3]{am}).
\end{th:cpctgenab}

%%%%%%%%%%%%%%%%%%%%%%%%%%%%%%%%%%%%%%%%%%%%%%%%%%%%%%%%%%%%%%%%%%%%%%%%%%%%%
%%%%%%%%%%%%%%%%%%%%%%%%%%%%%%%%%%%%%%%%%%%%%%%%%%%%%%%%%%%%%%%%%%%%%%%%%%%%%
\section{Solvable Lie groups}\label{se:solvlie}

The main result is

\begin{theorem}\label{th:solv}
  Let $\bG$ be a Lie group with 
  \begin{itemize}
  \item solvable connected component $\bG_0$;
  \item and finitely-generated component group $\bG/\bG_0$
  \end{itemize}
  and $\alpha\in H^2(\bG,\bS^1)$ a cohomology class. If $(\bG,\alpha)$ is type-I then it is robustly so.
\end{theorem}

The crushing majority of the requisite hard work is contained in \cite{puk}; we will recall just enough of that background to (one hopes) make sense of the proof. 

Let $\bH$ be a connected, simply-connected solvable Lie group (so denoted in order to distinguish it from the generic, possibly non-simply-connected $\bG$, and its more laboriously-named universal cover $\widetilde{\bG}$). \cite[Chapter IV]{puk} describes a bijection between the space of primitive ideals of the full $C^*$-algebra $C^*(\bH)$ and a space of what \cite[\S 4.9]{puk} refers to as {\it generalized orbits}. These are certain torus bundles over subsets of the dual
\begin{equation*}
  \fh^*:=\mathrm{Hom}(\fh,\bR),\ \fh:=\text{the Lie algebra }Lie(\bH). 
\end{equation*}
A very brief overview follows, which serves the dual purpose of fixing some notation. The reader is encouraged to also consult \cite[\S 4.2]{puk} for a very illuminating summary.

\begin{itemize}
\item Let $f\in \fh^*$ be a linear functional defined on the Lie algebra $\fh:=Lie(\bH)$ and $\bH_f$ its isotropy group with respect to the {\it coadjoint action}: the dual of the adjoint action of $\bH$ on $\fh$ by conjugation.
\item The commutator subgroup $\bH'=[\bH,\bH]$ is closed, nilpotent, connected and simply-connected (as follows, say, from \cite[\S II, ending remark]{chv-solv}) with Lie algebra $\fh'=[\fh,\fh]$.
\item The restriction of $f$ to the Lie algebra
  \begin{equation*}
    \fh_f := Lie(\bH_f) = Lie(\bH_{f,0})
  \end{equation*}
  of the identity component $\bH_{f,0}\le \bH_f$ vanishes on the derived subalgebra $\fh_f'$ (by invariance under the coadjoint action), so is a {\it Lie algebra} morphism $\fh_f\to \bR$. Because $\bH_{f,0}$ is connected and simply-connected, $f$ integrates to a unique Lie {\it group} morphism (i.e. character)
  \begin{equation*}
    \chi_f:\bH_{f,0}\to \bS^1
  \end{equation*}
  by \cite[\S III.6.1, Theorem 1]{bourb-lie-13}.
\item It is observed in \cite[\S 4.2, III]{puk} that the intersection $\bH'\cap \bG_f$ is connected and hence coincides with $\bH'\cap \bH_{f,0}$. It follows, then, that for $x,y\in \bH_f$ the commutator
  \begin{equation*}
    [x,y]:=xyx^{-1}y^{-1}\text{ belongs to }\bH_{f,0}.
  \end{equation*}

\item Coadjoint-invariance then also ensures that
  \begin{equation}\label{eq:omegaf}
    \bH_f\times \bH_f\ni (x,y)\xmapsto[]{\quad\omega_f\quad} \chi_f([x,y])\in \bS^1
  \end{equation}
  is a skew-symmetric bilinear form (the right-hand side makes sense by the previous point).
\item One can then define $\overline{\bH}_f$ as the kernel of that form:
  \begin{equation}\label{eq:hfbar}
    \overline{\bH}_f:=\ker \omega:=\{x\in \bH_f\ |\ \omega(x,-)\equiv 1\equiv \omega(-,x)\}. 
  \end{equation}
\item The character $\chi_f:\bH_f\to \bS^1$ always extends to $\overline{\bH}_f$, and the set of extensions is a {\it torsor} (i.e. free transitive space) over the torus
  \begin{equation}\label{eq:tftor}
    \bT_f:=\widehat{\overline{\bH}_f/\bH_{f,0}}
  \end{equation}
  (torus because, as explained in \cite[\S 4.2, III]{puk} again, the larger group $\bH_f/\bH_{f,0}$ is finitely-generated free abelian).
\end{itemize}

Now, the coadjoint action of $\bH$ on $\fh^*$ might not be smooth (i.e. \cite[Theorem 1]{glm}: the quotient space might not be $T_0$, or some coadjoint orbits might fail to be locally closed). \cite[\S 4.8]{puk} nevertheless introduces a smooth relation `$\sim$' on $\fh^*$, coarsening that induced by the coadjoint action (i.e. so that $\sim$-classes are unions of coadjoint $\bH$-orbits), and minimal with this property. Its classes are referred to there as {\it regularized} or {\it R-orbits}; each is the closure of any of the coadjoint orbits it contains.

Fix an R-orbit $\cO\subset \fh^*$. It can be shown that the set
\begin{equation*}
  \cL(\cO):=\{(f,\chi)\ |\ f\in \cO,\ \chi:\overline{\bH}_f\to \bS^1\text{ restricting to }\chi_f\text{ on }\bH_{f,0}\}
\end{equation*}
is a fibration over $\cO$, with fiber \Cref{eq:tftor} (in particular, the dimension of that torus depends only on $\cO$, not on $f\in \cO$). $\bH$ has a natural action on $\cL(\cO)$, the closures of its orbits are classes of an equivalence relation, and it is these closures that \cite[\S 4/9]{puk} refers to as {\it generalized orbits}.

The gist of \cite[Chapter IV]{puk} is then to describe a bijection
\begin{equation}\label{eq:pukbij}
  (\text{generalized orbits}) \longleftrightarrow (\text{primitive ideals of $C^*(\bH)$}). 
\end{equation}

We are here interested in possibly non-simply-connected groups, which would be of the form $\bH/\bD$ for discrete central $\bD\le \bH$. The central group $\bD$ is of course contained in $\bH_f$ and in fact $\overline{\bH_f}$ for any $f\in \fh^*$ whatsoever, so it makes sense to restrict characters of $\overline{\bH}_f/\bH_{f,0}$ to the image of $\bD$ therein. Running through the brief description of \Cref{eq:pukbij} given in \cite[\S 4.9]{puk}, it is not difficult to check the following slight amplification thereof (needed below):

\begin{fact}\label{fc:hd}
  Let $\bH$ be a connected, simply-connected solvable group and $\bD\le \bH$ a discrete central subgroup.

  The bijection \Cref{eq:pukbij} restricts to an identification between the primitive-ideal space of $C^*(\bH/\bD)$ and those generalized orbits consisting of elements $(f,\chi)\in \cL(\cO)$ for
  \begin{equation*}
    \chi:\overline{\bH}_f\to \bS^1,\ \chi|_{\bD}\equiv 1. 
  \end{equation*}
  In particular, the R-orbits $\cO$ featuring in this bijection are precisely those containing $f$ for which $\chi_f$ annihilates $\bD\cap \bH_{f,0}$. 
\end{fact}
\begin{proof}
  Most of this was sketched in the discussion preceding the statement. The only point perhaps worth elaborating is the last claim. Specifically, it should be clear upon perusal of the cited material of \cite[Chapter IV]{puk} why {\it only} the R-orbits $\cO$ in the statement correspond to primitive ideals of $C^*(\bH/\bD)$; the claim, however, is that {\it all} such R-orbits appear.

  Formally, the claim is that if $f\in \fh^*$ is such that $\chi_f$ annihilates $\bD\cap \bH_{f,0}$, then it can be extended to a character $\chi:\overline{\bH}_f\to \bS^1$ annihilating all of $\bD$. To verify this, consider the extensions
  \begin{equation}\label{eq:hha}
    \begin{tikzpicture}[auto,baseline=(current  bounding  box.center)]
      \path[anchor=base] 
      (0,0) node (ll) {$\{1\}$}
      +(1.5,.7) node (lu) {$\bH_{f,0}$}
      +(1.5,-.7) node (ld) {$\bS^1$}
      +(3.5,.8) node (mu) {$\overline{\bH}_f$}
      +(3.5,-.8) node (md) {$\bullet$}
      +(5.5,0) node (r) {$\bA$}      
      +(6.5,0) node (rr) {$\{1\}$}
      ;
      \draw[->] (ll) to[bend left=6] node[pos=.5,auto] {$\scriptstyle $} (lu);
      \draw[->] (ll) to[bend right=6] node[pos=.5,auto] {$\scriptstyle $} (ld);
      \draw[->] (lu) to[bend left=6] node[pos=.5,auto] {$\scriptstyle $} (mu);
      \draw[->] (ld) to[bend right=6] node[pos=.5,auto] {$\scriptstyle $} (md);
      \draw[->] (mu) to[bend left=6] node[pos=.5,auto] {$\scriptstyle $} (r);
      \draw[->] (md) to[bend right=6] node[pos=.5,auto] {$\scriptstyle $} (r);
      \draw[->] (r) to[bend right=0] node[pos=.5,auto] {$\scriptstyle $} (rr);
      \draw[->] (lu) to[bend right=0] node[pos=.5,auto] {$\scriptstyle \chi_f$} (ld);
      \draw[->] (mu) to[bend right=0] node[pos=.5,auto] {$\scriptstyle$} (md);
    \end{tikzpicture}
  \end{equation}
  where
  \begin{equation*}
    \bA:=\overline{\bH}_f/\bH_{f,0}\text{ is finitely-generated free abelian}.
  \end{equation*}
  The definition of $\overline{\bH}_{f}$ as the kernel of the bilinear form \Cref{eq:omegaf} implies that the group $\bullet$ is in fact abelian, and hence the bottom extension splits (non-canonically) as
  \begin{equation*}
    \bullet\cong \bS^1\oplus \bA. 
  \end{equation*}
  Now simply extend the identity character on $\bS^1$ to a character of all of $\bullet$ by supplementing it with a character of $\bA$ annihilating the image of
  \begin{equation*}
    \bD\le \overline{\bH}_{f}\to \overline{\bH}_{f}/\bH_{f,0}=\bA. 
  \end{equation*}
  The resulting character on $\bullet$ will then pull back along the longer vertical map in \Cref{eq:hha} to one on $\overline{\bH}_f$ that
  \begin{itemize}
  \item coincides with the original $\chi_f$ on $\bH_{f,0}$;
  \item and annihilates all of $\bD$.
  \end{itemize}
  This was precisely the claim to be proven, so we are done. 
\end{proof}

The following piece of terminology will make it slightly less awkward to speak about the coadjoint orbits or R-orbits attached to a possibly non-simply-connected $\bH/\bD$. 

\begin{definition}\label{def:orbpert}
  Let $\bH$ be a connected, simply-connected, solvable Lie group with Lie algebra $\fh:=Lie(\bH)$ and $\bD\subset \bH$ a discrete central subgroup.
  \begin{itemize}
  \item An element $f\in \fh^*$ is {\it relevant} or {\it pertinent} to $\bH/\bD$ if it belongs to an R-orbit featuring in the primitive-ideal-generalized orbit bijection of \Cref{fc:hd} for $C^*(\bH/\bD)$.
  \item The same terms apply to coadjoint or R-orbits themselves: they are {\it relevant} or {\it pertinent} to $\bH/\bD$ if they contain elements with this property.
  \item For $(\bH/\bD)$-pertinent/relevant (R- or coadjoint) orbits we might also simply say that these are orbits {\it of} $\bH/\bD$.
  \end{itemize}
\end{definition}

\pf{th:solv}
\begin{th:solv}
  We make a number of gradual simplifications.
  \begin{enumerate}[(1)]
  \item\label{item:10} {\bf Disposing of cohomology.} The cohomology class $\alpha$ of the statement gives a central extension
    \begin{equation*}
      \{1\}\to \bS^1\to \bullet\to \bG\to \{1\},
    \end{equation*}
    and we have to argue that if it is of type I then any finite extension thereof is as well. But note that the middle term $\bullet$ itself satisfies the hypothesis (has solvable identity component, etc.), so we can just work with {\it it} directly (in place of $\bG$).

    For the duration of the proof, then, we can safely ignore $\alpha$.
    
  \item\label{item:11} {\bf Connected $\bG$.} That is, we do not have to worry, here, about the discrete portion $\bG/\bG_0$.

    The preceding remedial discussion (including \Cref{fc:hd}) applies to the universal cover $\bH:=\widetilde{\bG}$ and a discrete central $\bD\le \bH$ with $\bG\cong \bH/\bD$. It applies equally to a smaller, cofinite
    \begin{equation*}
      \bD_{\bE}\le \bD,\ \bF:= \bD/\bD_{\bE}\text{ finite}
    \end{equation*}
    giving a quotient $\bE:=\bH/\bD_{\bE}$ fitting into a finite central extension
    \begin{equation*}
      \{1\}\to \bF\to \bE\to \bG\to \{1\}. 
    \end{equation*}
    \cite[\S 4.13, Proposition]{puk} precisely describes when the generalized orbits discussed in \Cref{fc:hd} are (i.e. correspond to primitive ideals) of type I: the relevant plain coadjoint orbits $\cO\subset \fh^*$ need to be locally closed, and the quotients
    \begin{equation*}
      \bH_f/\overline{\bH}_f,\ f\in \cO
    \end{equation*}
    have to be finite. Equivalently (and more conveniently for the present discussion), the bilinear form \Cref{eq:omegaf} must have finite image. We are assuming this is the case for $\bG$, and want to deduce the same of $\bE$.
    
    An R-orbit relevant to $\bE$ consists of $f$ whose $\chi_f$ annihilates $\bD_{\bE}\cap \bH_f$. Because the quotient
    \begin{equation*}
      \bF:=\bD/\bD_{\bE} 
    \end{equation*}
    is finite of order, say, $n:=|\bF|$, the rescaled $nf$ annihilates
    \begin{equation*}
      \bD\cap \bH_f = \bD\cap \bH_{nf}\quad (\text{noting that }\bH_f=\bH_{nf}). 
    \end{equation*}
    It thus follows that $nf$ features in the primitive-ideal classification for $\bG\cong \bH/\bD$. It follows, then, that its coadjoint orbit is locally closed, hence so is that of $f = \frac 1n (nf)$. This disposes, then, of the orbit-regularity half of the story.

    To address the finiteness of $\bH_f/\overline{\bH}_f$, observe that by the same argument
    \begin{equation*}
      \bH_{f}/\overline{\bH}_{nf} = \bH_{nf}/\overline{\bH}_{nf}
    \end{equation*}
    {\it is} finite (by the type-I assumption on $\bG$). But this amounts to the $n^{th}$ power
    \begin{equation*}
      \omega_f^n = \omega_{nf}:\bH_f\times \bH_f\to \bS^1
    \end{equation*}
    of \Cref{eq:omegaf} having finite image, whence the finite-image constraint on $\omega_f$ itself.
    
  \item\label{item:12} {\bf General case.} Suppose $\bG$ as in the statement is of type I, and
    \begin{equation*}
      \{1\}\to \bF\to \bE\to \bG\to \{1\}
    \end{equation*}
    is a finite central extension thereof.

    The quotient $\bG/\bG_0$ is type-I (being a quotient of a type-I group), so modulo harmless substitution of a cofinite subgroup of $\bG$ for the latter we can assume, via Thoma's theorem, that $\bG/\bG_0$ is in fact abelian. Additionally, the identity component $\bG_0$ is also of type I, being open in a type-I group \cite[Proposition 2.4]{kal2}. The preceding discussion (on the connected case) thus applies to it.

    The problem naturally splits into two special instances: lifting the type-I property along the quotient by $\bF/\bF\cap \bE_0$ first, and then lifting again along the (quotient by) remainder $\bF\cap \bE_0$. To keep the notational overload at a minimum, we consider the two cases separately:

    \begin{enumerate}[(a)]
    \item {\bf The central finite group $\bF$ intersects the identity component $\bE_0$ trivially.} Or: the covering $\bE\to \bG$ restricts to an isomorphism
      \begin{equation*}
        \bE_0\cong \bG_0.
      \end{equation*}
      We apply the Mackey machine (summarized, for instance, in \cite[Chapter I, Proposition 10.4]{am}) twice, to the normal subgroups
      \begin{equation*}
        \bG_0\trianglelefteq \bE\ \text{or}\ \bG.
      \end{equation*}
      The orbits of the induced actions of
      \begin{equation*}
        \bA:=\bG/\bG_0\quad\text{or}\quad \bA_{\bE}:=\bE/\bG_0
      \end{equation*}
      on the spectrum $\widehat{\bG_0}$ are the same. Because $\bG$ is of type I and $\bA$ is discrete abelian, \cite[\S II.8, Corollary to Theorem 9]{am} implies that those orbits are locally closed. For each $x\in\widehat{\bG_0}$ denoting by `$x$' subscripts the respective isotropy groups, we have a finite central extension
      \begin{equation*}
        \{1\}\to \bF\to \bA_{\bE,x}\to \bA_x\to \{1\}.
      \end{equation*}
      Furthermore, the Mackey obstruction class in $H^2(\bA_{\bE,x},\bS^1)$ is pulled back from the analogous class $\alpha_x\in H^2(\bA_{x},\bS^1)$, and \Cref{th:cpctgenab} applies to give
      \begin{equation*}
        (\bA_x,\alpha_x)\text{ type-I}\Rightarrow (\bA_{\bE,x},\alpha_x)\text{ type-I}.
      \end{equation*}
      The already-mentioned \cite[Chapter I, Proposition 10.4]{am} then finishes the proof that $\bE$ is of type I.

    \item {\bf $\bF$ is contained in the identity component $\bE_0$.} This time the covering $\bE\to \bG$ identifies the quotients
      \begin{equation}\label{eq:samea}
        \bA:=\bE/\bE_0\cong \bG/\bG_0
      \end{equation}
      and restricts to the identity component to give a finite central extension
      \begin{equation*}
        \{1\}\to \bF\to \bE_0\to \bG_0\to \{1\}.
      \end{equation*}      
      The discussion in part \Cref{item:11} of the present proof (pertaining to {\it connected} solvable groups) applies to this latter extension to ensure that $\bE_0$, at least, is of type I.

      Furthermore, recall the device used in that earlier argument to move between $\widehat{\bG_0}$ and $\widehat{\bE_0}$: identifying generalized orbits with elements of the spectrum as in \Cref{fc:hd}, if $(f,\chi)$ corresponds to an element (i.e. sits on a generalized orbit)
      \begin{equation*}
        x\in \widehat{\bE_0}
      \end{equation*}
      then we can associate to it the generalized orbit
      \begin{equation*}
        \widehat{\bG_0}\ni \varphi(x)\text{ generated by }(nf,\chi^n),\ n:=|\bF|. 
      \end{equation*}
      That $\varphi$ is equivariant for the conjugation actions of \Cref{eq:samea} is immediate, so that
      \begin{equation}\label{eq:axax}
        \bA_x\le \bA_{\varphi(x)}.
      \end{equation}
      Furthermore, the Mackey obstruction $\alpha_x\in H^2(\bA_x,\bS^1)$ relates to its analogue as follows:
      \begin{equation*}
        \alpha_x^n=\text{restriction of $\alpha_{\varphi(x)}$ along \Cref{eq:axax}}.
      \end{equation*}
      We are assuming $\bG$ (hence $\bG_0$) is of type I, hence so is the middle term in the central extension
      \begin{equation*}
        \{1\}\to \bS^1\to \bullet\to \bA_{\varphi(x)}\to \{1\}
      \end{equation*}
      attached to $\alpha_{\varphi(x)}$. But then by \cite[Proposition 2.4]{kal2} so is its open subgroup $\square$ in the top extension 
      \begin{equation*}
        \begin{tikzpicture}[auto,baseline=(current  bounding  box.center)]
          \path[anchor=base] 
          (0,0) node (ll) {$\{1\}$}
          +(1.5,0) node (l) {$\bS^1$}
          +(3.5,.5) node (mu) {$\square$}
          +(3.5,-.5) node (md) {$\bullet$}
          +(5.5,.5) node (ru) {$\bA_x$}
          +(5.5,-.5) node (rd) {$\bA_{\varphi(x)}$}
          +(7.5,0) node (r) {$\{1\}$}
          ;
          \draw[->] (ll) to[bend left=0] node[pos=.5,auto] {$\scriptstyle $} (l);
          \draw[->] (l) to[bend left=6] node[pos=.5,auto] {$\scriptstyle $} (mu);
          \draw[->] (l) to[bend right=6] node[pos=.5,auto,swap] {$\scriptstyle $} (md);
          \draw[->] (mu) to[bend left=0] node[pos=.5,auto] {$\scriptstyle $} (ru);
          \draw[->] (md) to[bend left=0] node[pos=.5,auto] {$\scriptstyle $} (rd);
          \draw[->] (ru) to[bend left=6] node[pos=.5,auto] {$\scriptstyle $} (r);
          \draw[->] (rd) to[bend right=6] node[pos=.5,auto,swap] {$\scriptstyle $} (r);
          \draw[right hook->] (mu) to[bend left=0] node[pos=.5,auto] {$\scriptstyle $} (md);
          \draw[right hook->] (ru) to[bend left=0] node[pos=.5,auto] {$\scriptstyle $} (rd);
        \end{tikzpicture}
      \end{equation*}
      induced by restriction along \Cref{eq:axax}.

      This means that
      \begin{itemize}
      \item $(\bA_x,\alpha_x^n)$ is of type I;
      \item so that $\alpha_x^n$ is of finite order by \cite[Chapter V, Lemma 6.1]{am}, because $\bA_x$ is finitely-generated abelian;
      \item meaning that $\alpha_x$ is also of finite order;
      \item and hence that $(\bA_x,\alpha_x)$ is of type I, by another application of \cite[Chapter V, Lemma 6.1]{am}.
      \end{itemize}
    \end{enumerate}
  \end{enumerate}
  This concludes the proof of \Cref{th:solv}.
\end{th:solv}

%%%%%%%%%%%%%%%%%%%%%%%%%%%%%%%%%%%%%%%%%%%%%%%%%%%%%%%%%%%%%%%%%%%%%%%%%%%%%
%%%%%%%%%%%%%%%%%%%%%%%%%%%%%%%%%%%%%%%%%%%%%%%%%%%%%%%%%%%%%%%%%%%%%%%%%%%%%
\section{Complements on algebraic hulls and central subgroups}\label{se:alghull}

%%%%%%%%%%%%%%%%%%%%%%%%%%%%%%%%%%%%%%%%%%%%%%%%%%%%%%%%%%%%%%%%%%%%%%%%%%%%%
\subsection{Splitting non-simply-connected groups}\label{subse:nonsc}

The discussion above, on finite central extensions, is of course only interesting when the groups we are extending are {\it not} simply-connected. Simple connectedness plays a key role in both \cite[\S III.1]{am} and \cite[\S IV.4]{ak}, in the following context.

\begin{construction}\label{con:alghull}
  Let $\bG$ be a connected, simply-connected solvable Lie group with Lie algebra $\fg:=Lie(\bG)$.
  \begin{itemize}
  \item Denote by $Ad:\bG\to GL(\fg)$ the adjoint representation.
  \item Let
    \begin{equation*}
      Ad(\bG)_{alg}\subset GL(\fg)
    \end{equation*}
    be the {\it algebraic hull} of $Ad(\bG)$: the smallest real algebraic group (i.e. definable by polynomial equations) containing $Ad(\bG)$.
  \item General algebraic-group theory (e.g. \cite[Chapter I, \S 4.4, Theorem]{brl}) we have a semidirect product decomposition
    \begin{equation}\label{eq:adus}
      Ad(\bG)_{alg} = Ad(\bG)_u \rtimes \bT
    \end{equation}
    where $Ad(\bG)_u$ consists of unipotent (i.e. eigenvalue-1) matrices on $\fg$ while $\bT$ consists of semisimple operators thereon.

    While $Ad(\bG)_u$ is uniquely determined as precisely the set of unipotent elements, $\bT$ is unique only up to conjugation; it is abelian, and in fact a maximal {\it torus} \cite[\S 8.5]{brl} in the broader sense familiar from algebraic geometry (i.e. not necessarily a product of circles; we use the phrase {\it algebraic torus} for clarity). We refer the reader to \cite[Theorem 10.6]{brl} for details.
  \item Composing the adjoint representation of $\bG$ with the surjection
    \begin{equation*}
      Ad(\bG)_{alg} = Ad(\bG)_u \rtimes \bT\to \bT
    \end{equation*}
    gives a morphism $\bG\to \bT$; denote its image by $\bS$.
  \item Because the connected automorphism group of the Lie algebra $\fg$ is algebraic in the sense above (e.g. \cite[Chapter IV, \S XV]{chv1}), $Ad(\bG)_{alg}$ acts as a group of automorphisms of $\fg$.
  \item The automorphisms $\bS$ of the Lie algebra $\fg$ are also automorphisms of the corresponding {\it simply-connected} Lie group $\bG$ \cite[\S III.6.1, Theorem 1]{bourb-lie-13}, so we can form the semidirect product
    \begin{equation}\label{eq:gssplit}
      \bG_{s}:=\bG\rtimes \bS. 
    \end{equation}
  \end{itemize}
  This is the {\it semisimple splitting} of $\bG$ of \cite[\S III.1]{am}
\end{construction}

\begin{remark}\label{re:axb}
  There is a claim in \cite[\S III.1]{am} that $Ad(\bG)_{alg}$ is connected; that source does not clarify whether what is meant is the {\it Zariski} or the standard (Lie-group) topology, but in the latter case connectedness does not hold in this generality: the adjoint representation of the ``$ax+b$'' (solvable, simply-connected) Lie group consisting of the transformations
  \begin{equation*}
    \bR\ni x\xmapsto[]{\quad \psi_{a,b}} ax+b\in \bR,\quad a\in \bR_{>0},\ b\in \bR
  \end{equation*}
  is, upon choosing the basis judiciously,
  \begin{equation}\label{eq:axb}
    \psi_{a,b}\mapsto
    \begin{pmatrix}
      a&b\\
      0&1
    \end{pmatrix}.
  \end{equation}
  The algebraic closure of this group of matrices is easy to determine: it consists of those \Cref{eq:axb} with {\it arbitrary} $a\in \bR^{\times}:=\bR\setminus\{0\}$, positive or negative. The issue is that positivity is not expressible algebraically, via polynomial vanishing.

  The algebraic hull $Ad(\bG)_{alg}$ will, however, always have finitely many (possibly more than one) connected components in its standard topology \cite[\S 24.6, (c) (i)]{brl}. 
\end{remark}

Simple connectedness features in that last step of \Cref{con:alghull}, where automorphisms of the Lie algebra are lifted over to automorphisms of $\bG$. In general, the question is whether the algebraic hull $Ad(\bG)_{alg}$ (or perhaps smaller subgroups of interest, e.g. its identity component) fix the discrete central subgroup $\bD\le \bG$ we are quotienting by, thus acting on the quotients $\bG/\bD$. The aim is to prove that this is the case:

\begin{theorem}\label{th:allfix}
  Let $\bG$ be a connected, simply-connected solvable Lie group. The algebraic hull $Ad(\bG)_{alg}$ fixes the center $Z(\bG)$ pointwise.
\end{theorem}

The first remark is that the connected component of the center, at least, poses no issues (so that $\bD\cap Z(\bG)_0$ is always pointwise invariant).

\begin{lemma}\label{le:centnilp}
  For any connected, simply-connected Lie group $\bG$ the algebraic hull $Ad(\bG)_{alg}$ fixes the intersection
  \begin{equation*}
    Z(\bG)_0 = Z(\bG)\cap N(\bG)
  \end{equation*}
  of the center and the nil-radical pointwise.
\end{lemma}
\begin{proof}
  We saw in \Cref{re:allaboutcenters} \Cref{item:6} that $Z(\bG)\cap N(\bG)$ coincides with the connected component $Z(\bG)_0$. An automorphism $\alpha$ of $\fg:=Lie(\bG)$ also operates on $\fn:=Lie(N(\bG))$ and $\bN:=N(\bG)$ and the exponential isomorphism \Cref{eq:expmap} is $\alpha$-equivariant. It follows that $\alpha$ fixes a non-trivial element $x\in \bN$ if and only if it fixes the unique 1-parameter subgroup of $\bN$ containing $x$.

  What all of this says is that the condition that $\alpha\in\mathrm{Aut}(\fg)$ leave $x\in \bN$ invariant is expressible algebraically, as the condition of operating trivially on a line in $\fg$ pointwise. When $x\in Z(\bG)_0$ the elements of $Ad(\bG)$ all satisfy this condition, hence so do those of its algebraic hull.  
\end{proof}

\begin{remark}\label{re:centonexp}
  Although in general, for a simply-connected solvable Lie group $\bG$, the exponential map need not be either injective or surjective \cite[Chapter II, Exercise B.4]{helg}, {\it central} subgroups of $\bG$ are better behaved: by \cite[Theorem 1]{chv-solv} every discrete central subgroup is contained in the image of an abelian Lie subalgebra of $Lie(\bG)$ through the exponential map.

  That abelian Lie subalgebra will not, in general, be unique: an {\it extended Mautner group} $\bM$ as defined in \cite[\S III.2]{am} is a semidirect product $\bR^b\rtimes \bR^a$ with the $\bR^b$ factor as the nilradical. It can be shown that the center is discrete and contained in the $\bR^a$ factor (and in fact isomorphic to $\bZ^{a-1}$: \Cref{le:genmaut}). One can easily obtain, now, multiple copies of $\bR^{a-1}$ containing that center as a lattice: start with a subgroup $\bR^{a-1}\subset \bR^a$, and conjugate by elements in $\bR^b$.
\end{remark}

\begin{proposition}\label{pr:ssfix}
  Let $\bG$ be a connected, simply-connected solvable Lie group. The semisimple elements in $Ad(\bG)_{alg}$ fix the center $Z(\bG)$ pointwise.
\end{proposition}
\begin{proof}
  We already know from \Cref{le:centnilp} that $Ad(\bG)_{alg}$ fixes the connected component $Z(\bG)_0$ pointwise. On the other hand,
  \begin{itemize}
  \item because $\bG$ is solvable, its center is of the form
    \begin{equation*}
      Z(\bG)_0\oplus \bZ^d
    \end{equation*}
    by \cite[Theorem 1]{chv-solv};
  \item and every semisimple element in $Ad(\bG)_{alg}$ is conjugate to an element of any one maximal algebraic torus torus $\bT$ that provides a decomposition \Cref{eq:adus} \cite[Theorem 10.6 (6)]{brl}.
  \end{itemize}
  Since $Ad(\bG)_{alg}$ operates on the center $Z(\bG)$ (the latter being a characteristic subgroup of $\bG$), it is enough to argue that {\it some} maximal algebraic torus $\bT$ fixes {\it some} discrete complement
  \begin{equation*}
    \bZ^d\le Z(\bG)\text{ of }Z(\bG)_0
  \end{equation*}
  pointwise.

  Consider such a central complement $\bZ^d$ to the connected center. By \cite[Theorem 1]{chv-solv} again, it is the image under the exponential map
  \begin{equation*}
    \exp:\left(\fg:=Lie(\bG)\right)\to \bG
  \end{equation*}
  of a lattice in some abelian Lie subalgebra $\fh\cong \bR^d\le \fg$. Now, $\exp(\fh)$ is a (plain, compact) torus in $Ad(\bG)$, so can be extended to a maximal algebraic torus $\bT$ by \cite[Theorem 10.6 (5)]{brl}. But $\bT$, being abelian, acts trivially on its own Lie algebra and hence also on $\fh$. This in turn implies that it fixes
  \begin{equation*}
    \bZ^d\subset \exp(\fh)
  \end{equation*}
  pointwise, as desired. 
\end{proof}

Finally:

\pf{th:allfix}
\begin{th:allfix}
  Certainly $\bG$ itself does, and the subset
  \begin{equation*}
    Ad(\bG)_s\subseteq Ad(\bG)_{alg}
  \end{equation*}
  consisting of semisimple elements (not a subgroup, in general!) does too, by \Cref{pr:ssfix}. To conclude, we will argue that $Ad(\bG)_{alg}$ is generated algebraically by $Ad(\bG)$ and $Ad(\bG)_s$. Naturally, the group so generated (call it $\bH$) contains $Ad(\bG)$ and will be contained in $Ad(\bG)_{alg}$, so the claim is that it must be algebraic.

  To see this, consider the {\it multiplicative Jordan decompositions} \cite[\S 4.2]{brl}
  \begin{equation}\label{eq:multjord}
    x=us,\ u\text{ unipotent },\ s\text{ semisimple },\ us=su
  \end{equation}
  of an element $x\in Ad(\bG)$. We have
  \begin{equation*}
    u\in Ad(\bG)_u,\ s\in Ad(\bG)_s
  \end{equation*}
  by \cite[Theorem 10.6]{brl}, and $\bH$ can also be recovered as the group generated by $Ad(\bG)_s$ and all $u$ appearing in such a decomposition.

  The group $Ad(\bG)$ is generated algebraically by $Ad(\bU)$ for a sufficiently small identity neighborhood $\bU\subset \bG$, and every element in $\bU$ will be contained in the image of the exponential map. It is thus enough to focus on decompositions \Cref{eq:multjord} for $x$ lying on 1-parameter
  \begin{equation*}
    \bR\ni t\mapsto \exp(tX)
  \end{equation*}
  subgroups of $Ad(\bG)$. But for these the multiplicative Jordan decomposition is of the form
  \begin{align*}
    t &\mapsto \exp(tU)\exp(tS) = \exp t(U+S)\\
    X &=U+S\text{ the additive Jordan decomposition of \cite[Proposition 4.2]{brl}}.
  \end{align*}
  It follows, then, that it is enough to consider the unipotent $u$ lying on 1-parameter groups $\exp tU$ for nilpotent $U$. But those groups are algebraic (\cite[\S 7.3, Remark]{brl} or \cite[Chapitre II, \S 13, Proposition 1]{chv2}) and all of them, together with $Ad(\bG)_s$, generate (algebraically) an algebraic group by \cite[Proposition 2.2]{brl}.
  
  % % it remains to observe that the elements of $Ad(\bG)_{alg}$ are all products of elements in $Ad(\bG)$ and semisimple elements:
  % % \begin{equation*}
  % %   Ad(\bG)_{alg} = \{xs\ |\ x\in Ad(\bG),\ s\in Ad(\bG)\text{ semisimple}\}.
  % % \end{equation*}
  % % Indeed, the set $\cS$ on the right contains a maximal algebraic torus $\bT$ by construction, so it is enough to argue that the unipotent part $Ad(\bG)_u$ of \Cref{eq:adus} is contained in $\cS$. To see this, observe that $Ad(\bG)_u$ is precisely the group consisting of the unipotent components $u$
  % % \begin{equation*}
  % %   x=us,\ u\text{ unipotent },\ s\text{ semisimple },\ us=su
  % % \end{equation*}
  % % in the {\it multiplicative Jordan decompositions} \cite[\S 4.2]{brl} of $x\in Ad(\bG)$ (for instance because that group is automatically algebraic, etc.), and those operators $u$ can indeed be recovered as
  % % \begin{equation*}
  % %   u=xs^{-1},\ x\in Ad(\bG),\ s\in Ad(\bG)\text{ semisimple}.
  % % \end{equation*}
  % % 
  
  This concludes the proof. 
\end{th:allfix}

%%%%%%%%%%%%%%%%%%%%%%%%%%%%%%%%%%%%%%%%%%%%%%%%%%%%%%%%%%%%%%%%%%%%%%%%%%%%%
\subsection{Central subgroups and their saturations}\label{subse:sat}

In reference to \Cref{re:centonexp}, it is perhaps worth noting that for a central subgroup $\bA$ of a connected, simply-connected solvable Lie group conjugation is essentially the only way to switch between Euclidean groups containing $\bA$:

\begin{proposition}\label{pr:1uptoconj}
  Let $\bG$ be a connected, simply-connected solvable Lie group and $\bA\le \bG$ a central subgroup.

  All closed, connected, abelian subgroups of $\bG$ of minimal dimension and containing $\bA$ are mutually conjugate.
\end{proposition}
\begin{proof}
  Denote by $\bN:=N(\bG)$ the nilradical of $\bG$. Each element of $\bA\cap \bN$ is contained in a unique 1-parameter subgroup of $\bN$ (because the latter's exponential is a diffeomorphism \cite[\S III.9.5, Proposition 13]{bourb-lie-13}), and those 1-parameter subgroups are then $\bG$-central and hence generate a connected $\bG$-central subgroup
  \begin{equation*}
    \bE\le \bN\le \bG. 
  \end{equation*}
  Since the ambient group is solvable and simply-connected, the closed, connected, abelian subgroups are precisely the Euclidean ones (i.e. copies of $\bR^d$) \cite[Theorem 1]{chv-solv}. Such a group contains $\bA\cap \bN$ if and only if it contains $\bE$, so nothing is lost by simply quotienting by that group throughout: make the substitutions
  \begin{equation*}
    \bG\mapsto \bG/\bE,\ \bN\mapsto \bN/\bE,\ \bA\mapsto \bA/\bA\cap \bE,\ \text{etc.}
  \end{equation*}
  The last group in particular is central and discrete (\Cref{le:dilp'}), so we may as well simplify matters and assume from the start that $\bA\cong \bZ^d$ is discrete and intersects $\bN$.

  Consider, now, the connected Lie group $\bG/\bA$. For a Euclidean group
  \begin{equation*}
    \bR^d\cong \bH\le \bG
  \end{equation*}
  containing $\bA$ as a lattice the corresponding quotient $\bH/\bA$ is a torus $\bT^d$, and must be maximal compact in $\bG/\bA$: the latter has the same homotopy type as any of its maximal compact subgroups $\bK$ (e.g. by \cite[\S 4.13, first Theorem]{mz}), so $\bK$ would have to have the same fundamental group
  \begin{equation*}
    \pi_1(\bG/\bA)\cong \bA\cong \bZ^d. 
  \end{equation*}
  But on the other hand $\bK$ is connected, compact and solvable, hence also abelian and a torus. All of this means it must be $d$-dimensional, so cannot contain $\bH/\bA$ properly.

  Now, all maximal compact subgroups of a locally compact connected group are mutually conjugate (\cite[\S 4.13, first Theorem]{mz} again), so we are done.
\end{proof}

\begin{lemma}\label{le:dilp'}
  Let $\bG$ be a connected Lie group and $\bD\le \bG$ a central subgroup. The image of $\bD$ in the quotient $\bG/N(\bG)$ is discrete.
\end{lemma}
\begin{proof}  
  If not, the closure of that image would be a closed Lie subgroup $\overline{\bH}$ of $\overline{\bG}$ with non-trivial connected component $\overline{\bH}_0$. The preimage $\bH\le \bG$ of $\overline{\bH}\le \overline{\bG}$ is generated topologically by the nilpotent group $\bN:=N(\bG)$ and the central group $\bD$, and is thus nilpotent. But then the preimage of $\overline{\bH}_0$ is normal, nilpotent and connected, contradicting the maximality of $\bN$.
\end{proof}

As a consequence of \Cref{pr:1uptoconj}, one can answer the natural question of whether the central subgroup $\bA$ therein can be recovered from the Euclidean groups which contain it. This is not quite possible, but one instead recovers a ``saturation'' thereof. To make sense of this, recall some language (e.g. \cite[\S 7, Definition and subsequent remarks]{kap}):

\begin{definition}\label{def:pure}
  A subgroup $\bA\le \bB$ of an abelian group is {\it pure} if, for every $a\in \bA$, whenever the equation
  \begin{equation*}
    a=nb,\ b\in \bB
  \end{equation*}
  has a solution, it also has one in $\bA$.

  The condition that $\bB/\bA$ be torsion-free implies purity. This is easily seen to be {\it equivalent} to purity when $\bB$ is of the form $\bR^d\times \bZ^e$, so in that case the {\it purification} of $\bA\le \bB$ will be the smallest intermediate group
  \begin{equation*}
    \bA\le \overline{\bA}\le \bB,\quad \bB/\overline{\bA}\text{ torsion-free}. 
  \end{equation*}
\end{definition}

\begin{proposition}\label{pr:purify}
  Let $\bG$ be a connected, simply-connected solvable Lie group and $\bA\le \bG$ a central subgroup. The following subgroups of $\bG$ all coincide:
  \begin{enumerate}[(a)]
  \item\label{item:15} the purification of $\bA$ in the center $Z(\bG)$;    
  \item\label{item:16} the intersection of all Euclidean subgroups of $\bG$ containing $\bA$;
  \item\label{item:17} the intersection of all conjugates of any one minimal-dimensional Euclidean subgroup of $\bG$ containing $\bA$.
  \end{enumerate}  
\end{proposition}
\begin{proof}
  A number of observations follow.
  \begin{enumerate}[(I)]
  \item {\bf \Cref{item:16} and \Cref{item:17} coincide.} Their very definition makes it plain that \Cref{item:17} can only be larger (since \Cref{item:16} is the intersection of {\it more} groups). On the other hand,
    \begin{equation}\label{eq:arz}
      \bA\cong \bR^d\times \bZ^e
    \end{equation}
    because central subgroups of $\bG$ are free abelian \cite[Theorem 1]{chv-solv}, and any Euclidean group containing it will in turn contain a minimal one, isomorphic to $\bR^{d+e}$. But all such are conjugate by \Cref{pr:1uptoconj}, so we also have \Cref{item:16} $\supseteq$ \Cref{item:17}.    
  \item {\bf \Cref{item:17} $\supseteq$ \Cref{item:15}.} The purification $\overline{\bA}$ of $\bA$ in the center consists of all elements that have finite order modulo $\bA$; what is being claimed here is that any minimal-dimensional Euclidean group that contains $\bA$ in fact contains $\overline{\bA}$.

    Because $\overline{\bA}/\bA$ is torsion, it is finite (every group in sight is finitely-generated). In particular, the two have the same {\it rank} $d+e$: they are both isomorphic to a group of the form $\bR^d\times \bZ^e$.
    
    Some minimal-dimensional $\bR^{d+e}$ will contain $\overline{\bA}$. It is also minimal-dimensional among the Euclidean groups containing just $\bA$, but all of {\it these} are mutually conjugate by \Cref{pr:1uptoconj}. Since $\overline{\bA}$ is central conjugations leave it fixed, so we are done.

  \item {\bf \Cref{item:15} $\supseteq$ \Cref{item:17}.} Or equivalently: if $\bA$ is pure then it coincides with the intersection $\bH$ of \Cref{item:17}.

    As discussed, we have a decomposition \Cref{eq:arz} and minimal Euclidean groups containing $\bA$ are isomorphic to $\bR^{d+e}$. $\bH$ is a closed subgroup of such a group, of equal rank $d+e$; it also carries an adjoint action by $\bG$, which is trivial on the ($\bG$-central) cocompact subgroup $A\le \bH$. But this implies that $\bH$ itself is central, and
    \begin{itemize}
    \item the assumed purity of $\bA\le Z(\bG)$;
    \item together with the cocompactness of $\bA\le \bH$
    \end{itemize}
    imply that in fact $\bA=\bH$.
  \end{enumerate}
  This concludes the proof of the result in its entirety.
\end{proof}

\section{Modes of failure}\label{se:fail}

Consider a connected, simply-connected solvable Lie group $\bH$. We recalled above, in the course of the proof of \Cref{th:solv}, the two-fold criterion given in \cite[\S 4.13, Proposition]{puk} for a primitive ideal of $C^*(\bG)$ to be of type I: there is
\begin{itemize}
\item a {\bf topological} condition, requiring that the ideal be attached to a locally-closed coadjoint orbit;
\item and a {\bf local} condition, requiring that the quotient $\bH_f/\overline{\bH}_f$ be finite for each $f$ on that coadjoint orbit (notation as in \Cref{eq:hfbar}).
\end{itemize}

\cite{dix-revet} gives an example of a connected, simply-connected, solvable Lie group $\bH$, not of type I, and a central copy
\begin{equation*}
  \bZ\cong \bD\subset \bH
\end{equation*}
of the integers such that $\bH/\bD$ {\it is} of type I. An examination of how that example functions will reveal that the quotient has the effect of enlarging the groups $\overline{\bH}_f$ so that the relevant quotients $\bH_f/\overline{\bH}_f$ become finite.

On the other hand, the coadjoint orbits of $\bH$ as a whole (not just those of $\bH/\bD$ in the sense of \Cref{def:orbpert}) are easily seen to all be locally closed. In other words, the example of \cite{dix-revet} operates entirely on the {\it local} type-I condition, leaving the {\it topological} condition unaffected. A first observation is that this is inevitable when taking quotients by central copies of the {\it integers}:

\begin{lemma}\label{le:zalone}
  Let $\bH$ be a connected, simply-connected, solvable Lie group and $\bD\subset \bH$ a discrete central subgroup isomorphic to $\bZ$.

  If $\bH$ has a non-locally-closed coadjoint orbit, then so does $\bH/\bD$. 
\end{lemma}
\begin{proof}
  Scaling elements of the dual $\fh^*$ (of the Lie algebra $\fh:=Lie(\bH)$) by non-zero reals is equivariant for the coadjoint action, preserves isotropy groups, orbits up to homeomorphism, etc. In particular, it will turn a non-locally-closed orbit into another such. But in the language of \Cref{fc:hd}, an $f\in \fh^*$ can be scaled by some $c\in \bR^{\times}$ so as to ensure that
  \begin{equation*}
    \chi_f|_{\bD\cap \bH_{f,0}}\equiv 1:
  \end{equation*}
  this is because $\bD\cong \bZ$, so we only need to annihilate a generator thereof.
\end{proof}

%%%%%%%%%%%%%%%%%%%%%%%%%%%%%%%%%%%%%%%%%%%%%%%%%%%%%%%%%%%%%%%%%%%%%%%%%%%%%
\subsection{Correcting for non-smoothness}\label{subse:corr}

The following examples feature prominently in \cite{am}. Following \cite[\S III.2, p.138]{am}:

\begin{definition}\label{def:genmaut}
  Let
  \begin{equation}\label{eq:rhorep}
    \rho:\bR^a\to O(\bR^b)\quad \text{(orthogonal group)} 
  \end{equation}
  be a linear representation such that
  \begin{enumerate}[(a)]
  \item\label{item:1} $\rho(\bR^a)$ is a (possibly non-closed) Lie subgroup isomorphic to $\bT^{a-1}\times \bR$;
  \item\label{item:2} while for any non-zero $\bR^a$-invariant $V\le \bR^b$ the image of $\bR^a$ in $O(\bR^b/V)$ is compact.
  \end{enumerate}
  The semidirect product $\bR^b\rtimes \bR^a$ induced by such an action is the {\it generalized (or extended) Mautner group} $\bM(\rho)$.

  We refer to a representation \Cref{eq:rhorep} meeting the requirements listed above as a {\it Mautner representation}.
\end{definition}

First, for the sake of completeness and because the claim was made in passing in \Cref{re:centonexp}:

\begin{lemma}\label{le:genmaut}
  The center of a generalized Mautner group $\bM(\rho)$ associated to a representation \Cref{eq:rhorep} is
  \begin{equation*}
    \bZ^{a-1}\cong \ker(\rho)\subset \bR^a\subset \bM:=\bM(\rho). 
  \end{equation*}
\end{lemma}
\begin{proof}
  That $\ker(\rho)$ is contained in the center is immediate, as is the fact that it must be free abelian of rank $a-1$ (by condition \Cref{item:1} of \Cref{def:genmaut}). We thus need the converse.

  The center $Z(\bM)$ intersects $\bR^b\le \bM$ trivially: any non-zero central element would span a line pointwise-fixed under $\rho$, contradicting condition \Cref{item:2} of \Cref{def:genmaut}. It thus follows that
  \begin{equation*}
    Z(\bM)\subset \bM\to \bR^a
  \end{equation*}
  is bijective onto its image. That image must consist of elements annihilated by $\rho$, and we are done.
\end{proof}

The importance of the generalized Mautner groups stems from the fact that in a sense they are emblematic of solvable groups for which the action on the spectrum of the nilradical fails to be smooth \cite[\S III.2, Theorem 2]{am}. More specifically, they are tailor-made to have the ``dangling'' copy of $\bR$ in the image $\bT^{a-1}\times \bR$ of \Cref{def:genmaut} \Cref{item:1} wrap around a torus densely. The ``bad orbits'' (as they are referred to in \cite[\S III.3]{am}, for instance) come about via this wrapping.

One way to restore the type-I property is to correct for the ill-behaved orbits by tracing them over with an action by a higher-dimensional torus:

\begin{definition}\label{def:compmaut}
  Consider a representation \Cref{eq:rhorep} with its associated extended Mautner group
  \begin{equation*}
    \bM(\rho) = \bR^b\rtimes \bR^a,
  \end{equation*}
  as in \Cref{def:genmaut}. The closure of the image $\rho(\bR^a)$ is an $d$-torus for some $d\ge a+1$, which in turn induces an action
  \begin{equation}\label{eq:rho'}
    \rho':\bR^{d}\to O(\bR^b)
  \end{equation}
  of $\bR^d$ on the same vector space $\bR^b$. Note that the images of $\rho$ and $\rho'$ commute pointwise.

  The {\it complete} Mautner group $\overline{\bM}:=\overline{\bM}(\rho)$ is the semidirect product
  \begin{equation*}
    \bM(\rho)\rtimes \bR^{d} \cong (\bR^b\rtimes \bR^a)\rtimes \bR^{d},
  \end{equation*}
  where $\bR^{d}$ acts trivially on the $\bR^a$ factor and via $\rho'$ on $\bR^b$.
\end{definition}

The upshot of the discussion above, on ``tracing over'' orbits, is that that process is sufficient to undo the type-I pathology.

\begin{proposition}\label{pr:compmaut1}
  Given an action $\rho$ as in \Cref{def:genmaut}, the complete Mautner group $\overline{\bM}:=\overline{\bM}(\rho)$ is of type I.
\end{proposition}
\begin{proof}

  % % It will be convenient, through a mild piece of notational abuse, to aggregate the mutually-commuting representations
  % % \begin{equation*}
  % %   \rho:\bR^a\to O(\bR^b)
  % %   \quad\text{and}\quad
  % %   \rho':\bR^{d}\to O(\bR^b)
  % % \end{equation*}
  % % of \Cref{def:genmaut} and \Cref{def:compmaut} into a single
  % % \begin{equation*}
  % %   \rho:\bR^a\times \bR^{d}\to O(\bR^b). 
  % % \end{equation*}
  % % By construction, note that the image of the latter coincides with the image of its restriction to just the $\bR^{d}$ factor.
  % % 
  % % 
  % % Now, the topological condition of \cite[\S 4.13, Proposition]{puk} holds: all coadjoint orbits of $\overline{\bM}$ are not only locally closed, but in fact closed. This is because
  % % \begin{itemize}
  % % \item the nilradical $N(\overline{\bM})=\bR^b$ operates with closed orbits anyway, as $\rho(\bR^b)$ consists of unipotent operators \cite[Proposition 4.10]{brl};
  % % \item while the quotient
  % %   \begin{equation*}
  % %     \bR^a\times \bR^{d}\cong \overline{\bM}/\bR^b
  % %   \end{equation*}
  % %   operates via its image
  % %   \begin{equation*}
  % %     \rho(\bR^a\times \bR^{d}) = \rho(\bR^{d}),
  % %   \end{equation*}
  % %   which is a torus (hence compact).    
  % % \end{itemize}
  % % 
  
  This is an immediate application of Mackey theory (\cite[Theorem 3.12]{mack-unit}, \cite[Chapter I, Proposition 10.4]{am}, etc.): the mutually-commuting representations
  \begin{equation*}
    \rho:\bR^a\to O(\bR^b)
    \quad\text{and}\quad
    \rho':\bR^{d}\to O(\bR^b)
  \end{equation*}
  of \Cref{def:genmaut} and \Cref{def:compmaut} aggregate into a single
  \begin{equation}\label{eq:bigrho}
    \rho:\bR^a\times \bR^{d}\to O(\bR^b),
  \end{equation}
  which then realizes $\overline{\bM}$ as a semidirect product
  \begin{equation*}
    \overline{\bM}\cong \bR^b\rtimes (\bR^a\times \bR^d). 
  \end{equation*}
  We are done, by \cite[Theorem 3.12]{mack-unit}, as soon as we observe that $\bR^b\le \overline{\bM}$ is regularly embedded: by construction, the image of \Cref{eq:bigrho} coincides with that of only the $\bR^d$ factor, which in turn is a torus. This means that the orbits of the action of $\overline{\bM}$ on $\widehat{\bR^b}$ are in fact compact, hence the conclusion.
\end{proof}

\Cref{pr:compmaut1} allows for natural examples of non-type-I (connected, simply-connected, solvable) Lie groups whose semidirect products with tori are of type I.

\begin{example}\label{ex:adm1}
  Consider a representation \Cref{eq:rhorep}, with its resulting extended and complete Mautner groups
  \begin{equation*}
    \bM:=\bM(\rho)
    \quad\text{and}\quad
    \overline{\bM}:=\overline{\bM}(\rho)
  \end{equation*}
  respectively, as in \Cref{def:genmaut,def:compmaut}.

  We know that $\bM$ itself is not of type I, as explained in \cite[\S III.2]{am} (via \cite[Chapter II, Corollary to Theorem 9]{am}). On the other hand, $\overline{\bM}$ is of type I by \Cref{pr:compmaut1}, hence so is its quotient by the kernel $\bZ^{d}$ of \Cref{eq:rho'}:
  \begin{equation*}
    \overline{\bM}/\bZ^{d} \cong \bM\rtimes (\bR^{d}/\bZ^{d})\cong \bM\rtimes \bT^{d}. 
  \end{equation*}
\end{example}

\begin{remark}
  The type-I group $\overline{\bM}\rtimes \bT^{d}$ is {\it admissible} in the sense of \cite[Chapter IV, Appendix]{ak}: a semidirect product of a connected, simply-connected solvable group $\bM$ by a compact abelian group acting faithfully on $\bM$ and trivially on $\bM/N(\bM)$.
  
  \cite[Chapter V, Introduction]{ak} mentions in passing that examples of non-type-I connected, simply-connected, solvable $\bG$ exist so that an admissible $\bG\rtimes \bK$ is of type I; \Cref{ex:adm1} is precisely that.
\end{remark}

%%%%%%%%%%%%%%%%%%%%%%%%%%%%%%%%%%%%%%%%%%%%%%%%%%%%%%%%%%%%%%%%%%%%%%%%%%%%%
%%%%%%%%%%%%%%%%%%%%%%%%%%%%%%%%%%%%%%%%%%%%%%%%%%%%%%%%%%%%%%%%%%%%%%%%%%%%%

%\bibliography{bib}{}

\begin{thebibliography}{10}

 
\bibitem{ak}
L.~Auslander and B.~Kostant.
\newblock Polarization and unitary representations of solvable {L}ie groups.
\newblock {\em Invent. Math.}, 14:255--354, 1971.

\bibitem{am}
Louis Auslander and Calvin~C. Moore.
\newblock Unitary representations of solvable {L}ie groups.
\newblock {\em Mem. Amer. Math. Soc.}, 62:199, 1966.

\bibitem{bekka-count}
Bachir Bekka.
\newblock The {P}lancherel formula for countable groups.
\newblock {\em Indag. Math. (N.S.)}, 32(3):619--638, 2021.

\bibitem{be}
Bachir Bekka and Siegfried Echterhoff.
\newblock On unitary representations of algebraic groups over local fields.
\newblock {\em Represent. Theory}, 25:508--526, 2021.

\bibitem{brl}
Armand Borel.
\newblock {\em Linear algebraic groups}, volume 126 of {\em Graduate Texts in
  Mathematics}.
\newblock Springer-Verlag, New York, second edition, 1991.

\bibitem{bourb-lie-13}
Nicolas Bourbaki.
\newblock {\em Lie groups and {L}ie algebras. {C}hapters 1--3}.
\newblock Elements of Mathematics (Berlin). Springer-Verlag, Berlin, 1998.
\newblock Translated from the French, Reprint of the 1989 English translation.

\bibitem{br-coh}
Kenneth~S. Brown.
\newblock {\em Cohomology of groups}, volume~87 of {\em Graduate Texts in
  Mathematics}.
\newblock Springer-Verlag, New York, 1994.
\newblock Corrected reprint of the 1982 original.

\bibitem{chv-solv}
Claude Chevalley.
\newblock On the topological structure of solvable groups.
\newblock {\em Ann. of Math. (2)}, 42:668--675, 1941.

\bibitem{chv2}
Claude Chevalley.
\newblock {\em Th\'{e}orie des groupes de {L}ie. {T}ome {II}. {G}roupes
  alg\'{e}briques}.
\newblock Actualit\'{e}s Scientifiques et Industrielles [Current Scientific and
  Industrial Topics], No. 1152. Hermann \& Cie, Paris, 1951.

\bibitem{chv1}
Claude Chevalley.
\newblock {\em Theory of {L}ie groups. {I}}, volume~8 of {\em Princeton
  Mathematical Series}.
\newblock Princeton University Press, Princeton, NJ, 1999.
\newblock Fifteenth printing, Princeton Landmarks in Mathematics.

\bibitem{de}
Anton Deitmar and Siegfried Echterhoff.
\newblock {\em Principles of harmonic analysis}.
\newblock Universitext. Springer, Cham, second edition, 2014.

\bibitem{dix-nil-1}
J.~Dixmier.
\newblock Sur les repr\'{e}sentations unitaires des groupes de {L}ie
  nilpotents. {I}.
\newblock {\em Amer. J. Math.}, 81:160--170, 1959.

\bibitem{dix-revet}
Jacques Dixmier.
\newblock Sur le rev\^{e}tement universel d'un groupe de {L}ie de type {I}.
\newblock {\em C. R. Acad. Sci. Paris}, 252:2805--2806, 1961.

\bibitem{dixc}
Jacques Dixmier.
\newblock {\em {$C\sp*$}-algebras}.
\newblock North-Holland Mathematical Library, Vol. 15. North-Holland Publishing
  Co., Amsterdam-New York-Oxford, 1977.
\newblock Translated from the French by Francis Jellett.

\bibitem{folland}
Gerald~B. Folland.
\newblock {\em A course in abstract harmonic analysis}.
\newblock Textbooks in Mathematics. CRC Press, Boca Raton, FL, second edition,
  2016.

\bibitem{glm}
James Glimm.
\newblock Locally compact transformation groups.
\newblock {\em Trans. Amer. Math. Soc.}, 101:124--138, 1961.

\bibitem{gk}
Elliot~C. Gootman and Robert~R. Kallman.
\newblock The left regular representation of a {$p$}-adic algebraic group is
  type {${\rm I}$}.
\newblock In {\em Studies in algebra and number theory}, volume~6 of {\em Adv.
  in Math. Suppl. Stud.}, pages 273--284. Academic Press, New York-London,
  1979.

\bibitem{helg}
Sigurdur Helgason.
\newblock {\em Differential geometry, {L}ie groups, and symmetric spaces},
  volume~34 of {\em Graduate Studies in Mathematics}.
\newblock American Mathematical Society, Providence, RI, 2001.
\newblock Corrected reprint of the 1978 original.

\bibitem{kal2}
Robert~R. Kallman.
\newblock Certain topological groups are type {I}. {II}.
\newblock {\em Advances in Math.}, 10:221--255, 1973.

\bibitem{kap}
Irving Kaplansky.
\newblock {\em Infinite abelian groups}.
\newblock University of Michigan Press, Ann Arbor, 1954.

\bibitem{mack-unit}
George~W. Mackey.
\newblock {\em The theory of unitary group representations}.
\newblock Chicago Lectures in Mathematics. University of Chicago Press,
  Chicago, Ill.-London, 1976.
\newblock Based on notes by James M. G. Fell and David B. Lowdenslager of
  lectures given at the University of Chicago, Chicago, Ill., 1955.

\bibitem{milnor-morse}
J.~Milnor.
\newblock {\em Morse theory}.
\newblock Annals of Mathematics Studies, No. 51. Princeton University Press,
  Princeton, N.J., 1963.
\newblock Based on lecture notes by M. Spivak and R. Wells.

\bibitem{mz}
Deane Montgomery and Leo Zippin.
\newblock {\em Topological transformation groups}.
\newblock Robert E. Krieger Publishing Co., Huntington, N.Y., 1974.
\newblock Reprint of the 1955 original.

\bibitem{moore-ext}
Calvin~C. Moore.
\newblock Extensions and low dimensional cohomology theory of locally compact
  groups. {I}, {II}.
\newblock {\em Trans. Amer. Math. Soc.}, 113:40--63; ibid. 113 (1964), 64--86,
  1964.

\bibitem{ped-aut}
Gert~K. Pedersen.
\newblock {\em {$C^*$}-algebras and their automorphism groups}.
\newblock Pure and Applied Mathematics (Amsterdam). Academic Press, London,
  2018.
\newblock Second edition of [ MR0548006], Edited and with a preface by S\o ren
  Eilers and Dorte Olesen.

\bibitem{puk}
Lajos Puk\'{a}nszky.
\newblock {\em Characters of connected {L}ie groups}, volume~71 of {\em
  Mathematical Surveys and Monographs}.
\newblock American Mathematical Society, Providence, RI, 1999.
\newblock With a preface by J. Dixmier and M. Duflo.

\bibitem{rob-gr}
Derek J.~S. Robinson.
\newblock {\em A course in the theory of groups}, volume~80 of {\em Graduate
  Texts in Mathematics}.
\newblock Springer-Verlag, New York, second edition, 1996.

\bibitem{rot}
Joseph~J. Rotman.
\newblock {\em An introduction to homological algebra}.
\newblock Universitext. Springer, New York, second edition, 2009.

\bibitem{steen}
Norman Steenrod.
\newblock {\em The topology of fibre bundles}.
\newblock Princeton Landmarks in Mathematics. Princeton University Press,
  Princeton, NJ, 1999.
\newblock Reprint of the 1957 edition, Princeton Paperbacks.

\bibitem{thoma}
Elmar Thoma.
\newblock \"{U}ber unit\"{a}re {D}arstellungen abz\"{a}hlbarer, diskreter
  {G}ruppen.
\newblock {\em Math. Ann.}, 153:111--138, 1964.

\bibitem{tt}
Fabio~Elio Tonti and Asger Törnquist.
\newblock A short proof of {T}homa's theorem on type {I} groups, 2019.
\newblock arXiv:1904.08313.

\end{thebibliography}
%\bibliographystyle{plain}

\addcontentsline{toc}{section}{References}

\Addresses

\end{document}